\documentclass{amsart}

\newtheorem{theorem}{Theorem}[section]
\newtheorem{thm}{Theorem}[section]
\newtheorem{lemma}{Lemma}[section]

\theoremstyle{definition}

\theoremstyle{remark}

\numberwithin{equation}{section}

\DeclareMathOperator{\tr}{Trace}
\DeclareMathOperator{\dv}{Div}
\DeclareMathOperator{\Hess}{Hess}

\begin{document}

\title{A new pinching theorem for complete self-shrinkers and its generalization}
\thanks{Research supported by the National Natural Science Foundation of China, Grant Nos. 11531012, 11371315, 11601478; and
the China Postdoctoral Science Foundation, Grant No. 2016M590530.}
\author{Li Lei}
\address{Center of Mathematical Sciences \\ Zhejiang University \\ Hangzhou 310027 \\ China}
\email{lei-li@zju.edu.cn}

\author{Hongwei Xu}
\address{Center of Mathematical Sciences \\ Zhejiang University \\ Hangzhou 310027 \\China}
\email{xuhw@zju.edu.cn}

\author{Zhiyuan Xu}
\address{Center of Mathematical Sciences \\ Zhejiang University \\ Hangzhou 310027 \\ China}
\email{srxwing@zju.edu.cn}


\date{}


\begin{abstract}
In this paper, we firstly verify that if $M$ is a
complete self-shrinker with polynomial volume growth in
$\mathbb{R}^{n+1}$, and
if the squared norm of the second fundamental form of $M$
satisfies $0\leq|A|^2-1\leq\frac{1}{18}$, then $|A|^2\equiv1$ and
$M$ is a round sphere or a cylinder. More generally, let $M$ be a complete $\lambda$-hypersurface with polynomial volume growth in $\mathbb{R}^{n+1}$ with $\lambda\neq0$. Then we prove that there exists an positive constant $\gamma$, such that if $|\lambda|\leq\gamma$ and the squared norm of the second fundamental form of $M$ satisfies
$0\leq|A|^2-\beta_\lambda\leq\frac{1}{18}$, then $|A|^2\equiv \beta_\lambda$, $\lambda>0$ and $M$ is a
cylinder. Here $\beta_\lambda=\frac{1}{2}(2+\lambda^2+|\lambda|\sqrt{\lambda^2+4})$.
\end{abstract}

\maketitle

\section{Introduction}

Suppose $X: M \rightarrow\mathbb{R}^{n+1}$ is an isometric
immersion. If the position vector $X$ evolves in the direction of
the mean curvature vector $\vec{H}$, this yields a
solution of mean curvature flow:
$$\left\{\begin{array}{llll} \frac{\partial}{\partial t}X(x, t)=\vec{H}(x, t),\,\,
x\in M,\\
X(x, 0)=X(x).
\end{array} \right.$$
An important class of solutions to the above mean curvature flow equations are
self-shrinkers \cite{Hu2}, which satisfy
\begin{equation}\label{selfshr}
H=-X^N,
\end{equation}
where $X^N$ is the projection of $X$ on the unit  normal vector
$\xi$, i.e., $X^N=\langle X,\xi\rangle$. We remark that some authors have a factor $\frac{1}{2}$ on the right-hand side of the defining equation for self-shrinkers.

Rigidity problems of self-shrinkers have been studied extensively.
As is known, there are close relations between self-shrinkers and
minimal submanifolds. But they are quite different on many aspects.
We refer the readers to \cite{GXXZ} for the rigidity problems of
minimal submanifolds. In \cite{AL}, Abresch--Langer classified all
smooth closed self-shrinker curves in $\mathbb{R}^2$. In 1990,
Huisken \cite{Hu2} proved that the only smooth closed self-shrinkers
with nonnegative mean curvature in $\mathbb{R}^{n+1}$ are round
spheres for $n\geq2$. Based on the work due to Huisken
\cite{Hu2,Hu3}, Colding--Minicozzi \cite{CM} proved that if $M$ is
an $n$-dimensional complete self-shrinker with nonnegative mean
curvature and polynomial volume growth in $\mathbb{R}^{n+1}$, then
$M$ is isometric to either the hyperplane $\mathbb{R}^n$, a round
sphere or a cylinder. In \cite{Bre}, Brendle verified that the round
sphere is the only compact embedded self-shrinker in $\mathbb{R}^3$
of genus zero.

In 2011, Le--Sesum \cite{LeSe} proved that any $n$-dimensional
complete self-shrinker with polynomial volume growth in
$\mathbb{R}^{n+1}$ whose squared norm of the second fundamental form
satisfies $|A|^2<1$ must be a hyperplane. Afterwards, Cao--Li
\cite{CaoLi} generalized this rigidity result to arbitrary
codimension and proved that if $M$ is an $n$-dimensional
complete self-shrinker with polynomial volume growth in
$\mathbb{R}^{n+q}$, and if $|A|^2\leq1$, then $M$ must be one of the
generalized cylinders. In 2014, Ding--Xin \cite{DX} proved the
following rigidity theorem for self-shrinkers in the Euclidean
space.
\begin{theorem}\label{thm1}
Let $M$ be an $n$-dimensional complete self-shrinker with polynomial
volume growth in $\mathbb{R}^{n+1}$. If the squared norm of the
second fundamental form satisfies $0\leq|A|^2-1\leq\frac{11}{500}$,
then $|A|^2\equiv1$ and $M$ is a round sphere or a cylinder.
\end{theorem}
In \cite{ChengWei2}, Cheng--Wei proved that if $M$ is an
$n$-dimensional complete self-shrinker with polynomial volume growth
in $\mathbb{R}^{n+1}$, and if $0\leq|A|^2-1\leq\frac{3}{7}$, where
$|A|$ is constant, then $|A|^2=1$.

Recently, Xu--Xu \cite{XX} improve Theorem \ref{thm1} and proved the
following rigidity theorem
\begin{theorem}\label{thm2}
Let $M$ be an $n$-dimensional complete self-shrinker with polynomial volume growth in $\mathbb{R}^{n+1}$.
If the squared norm of the second fundamental form satisfies
$0\leq|A|^2-1\leq\frac{1}{21}$, then $|A|^2\equiv1$ and $M$ is a round sphere or a cylinder.
\end{theorem}

In this paper, we firstly prove the following rigidity theorem for
self-shrinkers in the Euclidean space.
\begin{thm}\label{mthm1}
Let $M$ be an $n$-dimensional complete self-shrinker with polynomial volume growth in $\mathbb{R}^{n+1}$.
If the squared norm of the second fundamental form satisfies
$0\leq|A|^2-1\leq\frac{1}{18}$, then $|A|^2\equiv1$ and $M$ is
one of the following cases:\\
(i) the round sphere $\mathbb{S}^n(\sqrt{n})$;\\
(ii) the cylinder $\mathbb{S}^k(\sqrt{k})\times\mathbb{R}^{n-k},\,\, 1\leq k\leq n-1$.
\end{thm}

More generally, we consider the rigidity of $\lambda$-hypersurfaces.
The concept of $\lambda$-hypersurfaces was introduced independently
by Cheng--Wei \cite{ChengWei1} via the weighted volume-preserving
mean curvature flow and McGonagle--Ross \cite{McRo} via
isoperimetric type problem in a Gaussian weighted Euclidean space.
Precisely, the hypersurfaces of Euclidean space satisfying the
following equation are called $\lambda$-hypersurfaces:
\begin{equation}\label{lamhyp}
H=-X^N+\lambda,
\end{equation}
where $X^N$ is the projection of $X$ on the unit  normal vector
$\xi$ and $\lambda$ is a constant. In recent years, the rigidity of
$\lambda$-hypersurfaces has been investigated by several authors
\cite{COW,ChengWei1,Guang,XLX}. In \cite{Guang}, Guang showed that
if $M$ is a $\lambda$-hypersurface with polynomial volume growth in
$\mathbb{R}^{n+1}$, and if $|A|^2\leq\alpha_\lambda$, then $M$ must
be one of the generalized cylinders, where
$\alpha_\lambda=\frac{1}{2}(2+\lambda^2-|\lambda|\sqrt{\lambda^2+4})$.
In the second part of this paper, we prove the following second
pinching theorem for $\lambda$-hypersurfaces in the Euclidean space.
\begin{thm}\label{mthm2}
Let $M$ be an $n$-dimensional complete $\lambda$-hypersurface with polynomial volume growth in $\mathbb{R}^{n+1}$ with $\lambda\neq0$. There exists an positive constant $\gamma$, such that if $|\lambda|\leq\gamma$ and the squared norm of the second fundamental form satisfies
$0\leq|A|^2-\beta_\lambda\leq\frac{1}{18}$, then $|A|^2\equiv \beta_\lambda$, $\lambda>0$ and $M$ must be
the cylinder $\mathbb{S}\Big(\frac{\sqrt{\lambda^2+4}-|\lambda|}{2}\Big)\times\mathbb{R}^{n-1}$. Here $\beta_\lambda=\frac{1}{2}(2+\lambda^2+|\lambda|\sqrt{\lambda^2+4})$.
\end{thm}

\section{rigidity of self-shrinkers}\label{sec2}

\par Let $M$ be an $n$-dimensional complete hypersurface in $\mathbb{R}^{n+1}$.
We shall make use of the following convention on the range of indices:
                     $$1\leq i, j, k, \ldots\leq n.$$
We choose a local orthonormal frame field $\{e_1, e_2, \ldots, e_{n+1}\}$
near a fixed point $x\in M$ over $\mathbb{R}^{n+1}$ such that
$\{e_i\}_{i=1}^{n}$ are tangent to $M$ and $e_{n+1}$ equals to the unit  normal vector $\xi$.
Let $\{\omega_1, \omega_2, \ldots, \omega_{n+1}\}$ be the dual frame fields of $\{e_1, e_2, \ldots, e_{n+1}\}$.
Denote by $R_{ijkl}$, $A:=\sum\limits_{i,j}h_{ij}\omega_{i}\otimes\omega_{j}$, $H:=\tr A$ and $S:=\tr A^2$ the Riemann curvature tensor, the second fundamental form, the mean curvature and the squared norm of the second fundamental form of $M$, respectively. We denote the first, the
second and the third covariant derivatives of the second fundamental
form of $M$ by
$$\nabla A=\sum\limits_{i,j,k}h_{ijk}\,\,\omega_{i}\otimes\omega_{j}\otimes\omega_{k},$$
$$\nabla^2 A=\sum\limits_{i,j,k,l}h_{ijkl}\,\,\omega_{i}\otimes\omega_{j}\otimes\omega_{k}\otimes\omega_{l},$$
$$\nabla^3 A=\sum\limits_{i,j,k,l,m}h_{ijklm}\,\,\omega_{i}\otimes\omega_{j}\otimes\omega_{k}\otimes\omega_{l}\otimes\omega_{m}.$$
The Gauss and Codazzi equations are given by
\begin{equation}\label{1.1}
R_{ijkl}=h_{ik}h_{jl}-h_{il}h_{jk},
\end{equation}
\begin{equation}\label{1.2}
h_{ijk}=h_{ikj}.
\end{equation}
We have the Ricci identities on $M$
\begin{equation}\label{2.4.3}
h_{ijkl}-h_{ijlk}=\sum\limits_{m}h_{im}R_{mjkl}+\sum\limits_{m}h_{mj}R_{mikl},
\end{equation}
\begin{equation}\label{2.4.4}
h_{ijklm}-h_{ijkml}=\sum\limits_{r}h_{rjk}R_{rilm}+\sum\limits_{r}h_{irk}R_{rjlm}+\sum\limits_{r}h_{ijr}R_{rklm}.
\end{equation}
We choose a local orthonormal frame $\{e_i\}$ such that
$h_{ij}=\mu_i\delta_{ij}$ at $x$.
By the Gauss equation (\ref{1.1}) and the Ricci identity (\ref{2.4.3}), we have
\begin{equation}\label{1.7}
t_{ij}:=h_{ijij}-h_{jiji}=\mu_i\mu_j(\mu_i-\mu_j).
\end{equation}
Set $u_{ijkl}=\frac{1}{4}(h_{ijkl}+h_{lijk}+h_{klij}+h_{jkli})$. Then we have
\begin{eqnarray}\label{1.9}
\sum\limits_{i,j,k,l}(h_{ijkl}^{2}-u_{ijkl}^{2})
&\geq&\nonumber\frac{6}{16}\sum\limits_{i\neq j}[(h_{ijij}-h_{jiji})^2+(h_{jiji}-h_{ijij})^2]\\
&=&\frac{3}{4}G,
\end{eqnarray}
i.e.,
\begin{equation}\label{1.10}
|\nabla^2 A|^2\geq\frac{3}{4}G,
\end{equation}
where $G=\sum\limits_{i,j}t^{2}_{ij}=2(Sf_4-f^{2}_{3})$ and $f_k=\tr A^k=\sum\limits_{i}\mu^{k}_{i}$.
In \cite{CM}, Colding-Minicozzi introduced the linear operator
$$\mathcal{L}=\Delta-\langle X, \nabla(\cdot)\rangle=e^{\frac{|X|^2}{2}}\dv\Big(e^{-\frac{|X|^2}{2}}\nabla(\cdot)\Big).$$
They showed that $\mathcal{L}$ is self-adjoint respect to the measure $\rho \,d\mu$, where $\rho=e^{-\frac{|X|^2}{2}}$.

Let $M$ be a self-shrinker with polynomial volume growth. By a computation (see \cite{DX,XX}), we have following equalities
\begin{equation}\label{1.3}
\mathcal{L}|A|^2=2|A|^2-2|A|^4+2|\nabla A|^2,
\end{equation}
\begin{equation}\label{1.4}
|\nabla S|^2=\frac{1}{2}\mathcal{L}S^2+2S^2(S-1)-2S|\nabla A|^2,
\end{equation}
\begin{equation}\label{1.5}
|\nabla^2 A|^2=\frac{1}{2}\mathcal{L}|\nabla A|^2+(|A|^2-2)|\nabla A|^2+3(B_1-2B_2)+\frac{3}{2}|\nabla S|^2,
\end{equation}
\begin{equation}\label{1.8}
\int_M(B_1-2B_2)\rho
\,d\mu=\int_M\Big(\frac{1}{2}G-\frac{1}{4}|\nabla S|^2\Big)\rho
\,d\mu,
\end{equation}
where $B_1=\sum\limits_{i,j,k,l,m}h_{ijk}h_{ijl}h_{km}h_{ml}$ and
$B_2=\sum\limits_{i,j,k,l,m}h_{ijk}h_{klm}h_{im}h_{jl}$.

Now we are in a position to prove our rigidity theorem for self-shrinkers in the Euclidean space.
\begin{proof}[Proof of Theorem \ref{mthm1}]
From (\ref{1.10}), (\ref{1.5}) and (\ref{1.8}), we have
\begin{eqnarray}\label{2.1}
& &\nonumber\int_M(B_1-2B_2)\rho \,d\mu\\
&=&\int_M\Big(\frac{1}{2}G-\frac{1}{4}|\nabla S|^2\Big)\rho \,d\mu\nonumber\\
&\leq&\nonumber\int_M\left(\frac{2}{3}|\nabla^2 A|^2-\frac{1}{4}|\nabla S|^2\right)\rho \,d\mu\\
&=&\int_M\left[\frac{2}{3}(S-2)|\nabla A|^2+2(B_1-2B_2)+\frac{3}{4}|\nabla S|^2\right]\rho \,d\mu.
\end{eqnarray}
This implies that
\begin{equation}\label{2.2}
\int_M(B_1-2B_2)\rho \,d\mu\geq\int_M\left[\frac{2}{3}(2-S)|\nabla A|^2-\frac{3}{4}|\nabla S|^2\right]\rho \,d\mu.
\end{equation}
By Lemma 4.2 in \cite{DX} and Young's inequality, for $\sigma>0$, we have
\begin{eqnarray}\label{2.3}
3(B_1-2B_2)&\leq&\nonumber(S+C_1G^{1/3})|\nabla A|^2\\
&\leq&S|\nabla A|^2+\frac{1}{3}C_1\sigma^2G+\frac{2}{3}C_1\sigma^{-1}|\nabla A|^3,
\end{eqnarray}
where $C_1=\frac{2\sqrt{6}+3}{\sqrt[3]{21\sqrt{6}+103/2}}$.
Notice that
\begin{equation}\label{2.4}
-\int_M \left\langle \nabla|\nabla A|,\nabla S\right\rangle\rho\,d\mu=\int_M|\nabla A|\mathcal{L} S\rho\,d\mu.
\end{equation}
This together with (\ref{1.3}) implies
\begin{eqnarray}\label{2.5}
\int_M|\nabla A|^3\rho\,d\mu&=&\nonumber\int_M\left(\frac{1}{2}\mathcal{L}S-S+S^2\right)|\nabla A|\rho\,d\mu\\
&=&\nonumber\int_M\left[(S^2-S)|\nabla A|-\frac{1}{2} \left\langle \nabla|\nabla A|,\nabla S\right\rangle\right]\rho\,d\mu\\
&\leq&\int_M\left[(S^2-S)|\nabla A|+\epsilon|\nabla^2 A|^2 +\frac{1}{16\epsilon} |\nabla S|^2\right]\rho\,d\mu.
\end{eqnarray}
From (\ref{1.5}), (\ref{1.8}), (\ref{2.3}), (\ref{2.5}), we have
\begin{eqnarray}\label{2.6}
& &\nonumber3\int_M (B_1-2B_2)\rho\,d\mu\\
&\leq&\nonumber\int_M \left(S|\nabla A|^2+\frac{1}{3}C_1\sigma^2G+\frac{2}{3}C_1\sigma^{-1}|\nabla A|^3\right)\rho\,d\mu\\
&\leq&\nonumber \int_M \left(S|\nabla A|^2+\frac{1}{3}C_1\sigma^2G\right)\rho\,d\mu\\
& &\nonumber+\frac{2}{3}C_1\sigma^{-1}\int_M\left[(S^2-S)|\nabla A|+\epsilon|\nabla^2 A|^2+\frac{1}{16\epsilon} |\nabla S|^2\right]\rho\,d\mu\\
&=&\nonumber \int_M \left[S|\nabla A|^2+\frac{2}{3}C_1\sigma^2\left(B_1-2B_2+\frac{1}{4}|\nabla S|^2\right)\right]\rho\,d\mu\\
& &\nonumber+\frac{2}{3}C_1\sigma^{-1}\int_M\left[(S^2-S)|\nabla A|+\frac{1}{16\epsilon} |\nabla S|^2\right]\rho\,d\mu\\
& &+\frac{2}{3}C_1\sigma^{-1}\epsilon\int_M\left[(S-2)|\nabla A|^2+3(B_1-2B_2)+\frac{3}{2}|\nabla S|^2\right]\rho\,d\mu.
\end{eqnarray}
Thus, we obtain
\begin{eqnarray}\label{2.7}
& &\nonumber3\theta\int_M (B_1-2B_2)\rho\,d\mu\\
&\leq&\nonumber\int_M \left[S+\frac{2}{3}C_1\sigma^{-1}\epsilon(S-2)\right]|\nabla A|^2\rho\,d\mu\\
& &\nonumber +\left(\frac{1}{6}C_1\sigma^2+C_1\sigma^{-1}\epsilon+\frac{1}{24\epsilon}C_1\sigma^{-1}\right)\int_M |\nabla S|^2\rho\,d\mu\\
& &+\frac{2}{3}C_1\sigma^{-1}\int_M(S^2-S)|\nabla A|\rho\,d\mu,
\end{eqnarray}
where $\theta=1-\left(\frac{2}{9} C_1 \sigma^2 + \frac{2}{3} C_1 \sigma^{-1}\epsilon \right)$.
We restrict $\sigma$ and $\epsilon$ such that $\theta \geq 0$.

Combining (\ref{2.2}) and (\ref{2.7}), we have
\begin{eqnarray}\label{2.8}
0&\leq&\nonumber\int_M \left[S+(\frac{2}{3}C_1\sigma^{-1}\epsilon+2\theta)(S-2)\right]|\nabla A|^2\rho\,d\mu\\
& &\nonumber + \left(\frac{1}{6}C_1\sigma^2+C_1\sigma^{-1}\epsilon+\frac{1}{24\epsilon}C_1\sigma^{-1}+\frac{9}{4}\theta\right)\int_M |\nabla S|^2\rho\,d\mu\\
& &+\frac{2}{3}C_1\sigma^{-1}\int_M(S^2-S)|\nabla A|\rho\,d\mu.
\end{eqnarray}
To simplify the notation, we put
$$L_1:=\frac{2}{3}C_1\sigma^{-1}\epsilon+2\theta,$$
$$L_2:=\frac{1}{6}C_1\sigma^2+C_1\sigma^{-1}\epsilon+\frac{1}{24\epsilon}C_1\sigma^{-1}+\frac{9}{4}\theta.$$
Then (\ref{2.8}) is reduced to
\begin{eqnarray}\label{2.8.1}
0&\leq&\nonumber\int_M \left[S+L_1 (S-2)\right]|\nabla A|^2\rho\,d\mu\\
& &+ L_2 \int_M |\nabla S|^2\rho\,d\mu
+\frac{2}{3}C_1\sigma^{-1}\int_M(S^2-S)|\nabla A|\rho\,d\mu.
\end{eqnarray}
When $0\leq S-1\leq\delta$, we have
\begin{eqnarray}\label{2.9}
\frac{1}{2}\int_M |\nabla S|^2\rho\,d\mu&=&\nonumber\int_MS(S-1)^2\rho\,dM-\int_M(S-1)|\nabla A|^2\rho\,d\mu\\
        &\leq&\int_M(1-S+\delta)|\nabla A|^2\rho\,d\mu.
\end{eqnarray}
For $\kappa>0$, we have
\begin{eqnarray}\label{2.10}
\int_MS(S-1)|\nabla A|\rho\,dM&\leq&\nonumber2(1+\delta)\kappa\int_MS(S-1)\rho\,d\mu\\
& &\nonumber+\frac{1}{8(1+\delta)\kappa}\int_MS(S-1)|\nabla A|^2\rho\,d\mu\\
&\leq&\nonumber2(1+\delta)\kappa\int_M |\nabla A|^2\rho\,d\mu\\
& &+\frac{1}{8\kappa}\int_M(S-1)|\nabla A|^2\rho\,dM.
\end{eqnarray}
Substituting (\ref{2.9}) and (\ref{2.10}) into (\ref{2.8.1}), we obtain
\begin{eqnarray}\label{2.11}
0&\leq&\nonumber\int_M [(1+L_1)(S-1)+1-L_1]|\nabla A|^2\rho\,d\mu\\
& &\nonumber+2L_2\int_M(1-S+\delta)|\nabla A|^2\rho\,d\mu\\
& &\nonumber+\frac{4}{3}C_1\sigma^{-1}(1+\delta)\kappa\int_M |\nabla A|^2\rho\,d\mu\\
& &\nonumber+\frac{1}{12\kappa}C_1\sigma^{-1}\int_M(S-1)|\nabla A|^2\rho\,dM\\
&=&\nonumber\int_M \left(1+L_1-2L_2+\frac{1}{12\kappa}C_1\sigma^{-1}\right)(S-1)|\nabla A|^2\rho\,d\mu\\
& &+\int_M \left[1-L_1+\frac{4}{3}C_1\sigma^{-1}\kappa+\left(\frac{4}{3}C_1\sigma^{-1}\kappa+2L_2\right)\delta\right]|\nabla A|^2\rho\,d\mu.
\end{eqnarray}
Let $\sigma=0.616$, $\epsilon=0.0577$ and $\kappa=0.0434$. By a
computation, we have
$$\theta >0,\,\,\,1+L_1-2L_2+\frac{1}{12\kappa}C_1\sigma^{-1}<0,$$
$$1-L_1+\frac{4}{3}C_1\sigma^{-1}\kappa<-0.452,$$
$$\frac{4}{3}C_1\sigma^{-1}\kappa+2L_2<8.03.$$
We take $\delta=1/18$. Then the coefficients of the integrals in
(\ref{2.11}) are both negative. Therefore, we have $|\nabla
A|\equiv0$ and $S\equiv1$, i.e., $M$ either the round sphere
$\mathbb{S}^n(\sqrt{n})$, or the cylinder
$\mathbb{S}^k(\sqrt{k})\times\mathbb{R}^{n-k},\,\, 1\leq k\leq n-1$.
\end{proof}

\section{rigidity of $\lambda$-hypersurfaces}

\par Let $M$ be an $n$-dimensional complete $\lambda$-hypersurface with polynomial volume growth in $\mathbb{R}^{n+1}$. We adopt the same notations as in Section \ref{sec2}. To simplify the computation, we choose local frame $\{e_i\}$, such that $\nabla_{e_i} e_j=0$ at $p\in M$, i.e., $\overline{\nabla}_{e_i}e_j=h_{ij}\xi$, and $h_{ij}=\mu_i\delta_{ij}$.
Then we have
\begin{equation}\label{2.4.5}
\nabla_{e_i}H=-\nabla_{e_i}\langle X,\xi\rangle=h_{ik}\langle X, e_k\rangle,
\end{equation}
and
\begin{eqnarray}\label{2.4.6}
\Hess H(e_i, e_j)&=&\nonumber-\nabla_{e_i}\nabla_{e_j}\langle X, \xi\rangle\\
&=&h_{ijk}\langle X, e_k\rangle+h_{ij}-(H-\lambda)h_{ik}h_{kj}.
\end{eqnarray}
Taking $f_k=\tr A^k=\sum\limits_{i}\mu^{k}_{i}$, we obtain
\begin{eqnarray}\label{2.4.7}
\mathcal{L}|A|^2&=&\nonumber \Delta|A|^2-\langle X, \nabla|A|^2\rangle\\
&=&\nonumber 2\sum\limits_{i,j}h_{ij}\Delta h_{ij}+2|\nabla A|^2-2\sum\limits_{i,j,k}h_{ij}h_{ijk}\langle X, e_k\rangle\\
&=&\nonumber 2(\sum\limits_{i,j}h_{ij}\nabla_{e_i}\nabla_{e_j}H+H\sum\limits_{i,j,k}h_{ij}h_{jk}h_{ki}-|A|^4)\\
& &\nonumber+2|\nabla A|^2-2\sum\limits_{i,j,k}h_{ij}h_{ijk}\langle X, e_k\rangle\\
&=&2|A|^2-2|A|^4+2\lambda f_3+2|\nabla A|^2.
\end{eqnarray}

\begin{proof}[Proof of Theorem \ref{mthm2}]
Putting $F_\lambda=|A|^4-|A|^2-\lambda f_3$, we have
\begin{equation}\label{2.4.7-1}
\int_M F_\lambda \rho dM=\int_M |\nabla A|^2 \rho dM.
\end{equation}
We also have
\begin{equation}\label{2.4.7-2}
|\nabla S|^2=\frac{1}{2}\mathcal{L}S^2+2SF_\lambda-2S|\nabla A|^2.
\end{equation}
Notice that
\begin{equation}\label{2.4.7-3}
F_\lambda\geq|A|^2(|A|^2-1-|\lambda|\cdot|A|).
\end{equation}
If $|A|^2\geq\beta_\lambda:=\frac{1}{2}(2+\lambda^2+|\lambda|\sqrt{\lambda^2+4})$, then $F_\lambda\geq0$.
Moreover, if $F_\lambda=0$, then $|A|^2=\beta_\lambda$. Denote by $\alpha_\lambda=\frac{1}{2}(2+\lambda^2-|\lambda|\sqrt{\lambda^2+4})$.
When $\beta_\lambda\leq|A|^2\leq\beta_\lambda+\delta$, we have the following upper bound for $F_\lambda$.
\begin{eqnarray}\label{2.4.7-4}
F_\lambda&\leq&\nonumber|A|^2(|A|^2-1+|\lambda|\cdot|A|)\\
&=&\nonumber|A|^2(|A|+\sqrt{\beta_\lambda})(|A|+\sqrt{\alpha_\lambda})^{-1}(|A|^2-\beta_\lambda+q_\lambda)\\    &\leq&\nonumber|A|^2(|A|^2-\beta_\lambda+q_\lambda)\Big(1+\frac{\sqrt{\beta_\lambda}-\sqrt{\alpha_\lambda}}{\sqrt{\beta_\lambda}+\sqrt{\alpha_\lambda}}\Big)\\
&\leq&|A|^2(|A|^2-\beta_\lambda+r_\lambda),
\end{eqnarray}
where $q_\lambda=|\lambda|\sqrt{\lambda^2+4}$ ,  $r_\lambda=q_\lambda+\lambda^2+\frac{|\lambda|\delta}{\sqrt{\lambda^2+4}}$.

For $|\nabla^2 A|$ and the integral of $B_1-2B_2$, we obtain the following lemma.
\begin{lemma}\label{lem1}
If $M$ is a $\lambda$-hypersurface of $\mathbb{R}^{n+1}$, then we have:\\
(i) $|\nabla^2 A|^2=\frac{1}{2}\mathcal{L}|\nabla A|^2+(|A|^2-2)|\nabla A|^2+3(B_1-2B_2)+\frac{3}{2}|\nabla S|^2-3\lambda C$,\\
(ii) $\int_M (B_1-2B_2)\rho dM=\int_M(\frac{1}{2}G-\frac{1}{4}|\nabla S|^2)\rho dM$,\\
where $B_1=\sum\limits_{i,j,k,l,m}h_{ijk}h_{ijl}h_{km}h_{ml}$, $B_2=\sum\limits_{i,j,k,l,m}h_{ijk}h_{klm}h_{im}h_{jl}$,
    $C=\sum\limits_{i,j,k,l}h_{ijk}h_{ijl}h_{kl}$, $G=\sum\limits_{i,j}t^{2}_{ij}=2(Sf_4-f^{2}_{3})$ and $t_{ij}=h_{ijij}-h_{jiji}=\mu_i\mu_j(\mu_i-\mu_j)$.
\end{lemma}
\begin{proof}
(i) Applying Ricci identities (\ref{2.4.3}) and (\ref{2.4.4}), we have
\begin{eqnarray}\label{lem1.2}
\Delta
h_{ijk}&=&h_{ijkll}=(h_{ijlk}+h_{ir}R_{rjkl}+h_{rj}R_{rikl})_l\nonumber\\
&=&h_{ijllk}+h_{rjl}R_{rikl}+h_{irl}R_{rjkl}+h_{ijr}R_{rlkl}+(h_{ir}R_{rjkl}+h_{rj}R_{rikl})_l\nonumber\\
&=&(h_{ljli}+h_{lr}R_{rjil}+h_{rj}R_{rlil})_k+h_{rjl}R_{rikl}+h_{irl}R_{rjkl}+h_{ijr}R_{rlkl}\nonumber\\
&&+h_{irl}R_{rjkl}+h_{rjl}R_{rikl}+h_{ir}(R_{rjkl})_l+h_{rj}(R_{rikl})_l\nonumber\\
&=&H_{jik}+h_{rkl}R_{rjil}+h_{rjk}R_{rlil}+2h_{rjl}R_{rikl}+2h_{ril}R_{rjkl}+h_{rij}R_{rlkl}\nonumber\\
&&+h_{ir}(R_{rjkl})_l+h_{rj}(R_{rikl})_l+h_{lr}(R_{rjil})_k+h_{rj}(R_{rlil})_k.
\end{eqnarray}
It follows from (\ref{2.4.6}) that
\begin{equation}\label{lem1.3}
H_{ji}=h_{jli}\langle X,e_l\rangle+h_{ij}+(\lambda-H)h_{il}h_{jl}.
\end{equation}
Since $\langle X,\xi\rangle=\lambda-H$, we compute the covariant derivative of $H_{ji}$
\begin{eqnarray}\label{lem1.4}
H_{jik}&=&h_{jlik}\langle X,e_l\rangle+h_{jli}\langle
e_k,e_l\rangle+h_{jli}\langle X,\overline{\nabla}_{e_k}e_l
\rangle+h_{ijk}\nonumber\\
&&-H_kh_{il}h_{jl}+(\lambda-H)(h_{ikl}h_{jl}+h_{il}h_{jkl})\nonumber\\
&=&h_{jlik}\langle X,e_l\rangle+2h_{ijk}-H_kh_{il}h_{jl}\nonumber\\
&&+(\lambda-H)(h_{il}h_{jkl}+h_{jl}h_{ikl}+h_{kl}h_{jli}).
\end{eqnarray}
Combining (\ref{lem1.2}) and (\ref{lem1.4}), we have
\begin{eqnarray}\label{lem1.5}
\Delta h_{ijk}&=&h_{jlik}\langle
X,e_l\rangle+2h_{ijk}-h_{il}h_{jl}H_{k}\nonumber\\
&&+(\lambda-H)(h_{il}h_{jlk}+h_{jl}h_{ikl}+h_{kl}h_{jli})\nonumber\\
&&+h_{rkl}R_{rjil}+h_{rjk}R_{rlil}+2h_{rjl}R_{rikl}+2h_{ril}R_{rjkl}+h_{rij}R_{rlkl}\nonumber\\
&&+h_{ir}(R_{rjkl})_l+h_{rj}(R_{rikl})_l+h_{lr}(R_{rjil})_k+h_{rj}(R_{rlil})_k.
\end{eqnarray}
The Gauss equation (\ref{2.4}) imples
\begin{eqnarray}\label{lem1.6}
&& h_{ijk}(h_{rkl}R_{rjil}+h_{rjk}R_{rlil}+2h_{rjl}R_{rikl}+2h_{ril}R_{rjkl}+h_{rij}R_{rlkl}\nonumber\\
&&+h_{ir}(R_{rjkl})_l+h_{rj}(R_{rikl})_l+h_{lr}(R_{rjil})_k+h_{rj}(R_{rlil})_k)\nonumber\\
&=& h_{ijk}(6h_{rkl}h_{ri}h_{jl}-6h_{rkl}h_{rl}h_{ij}+3h_{rij}h_{rk}h_{ll}-3h_{rij}h_{rl}h_{kl}\nonumber\\
&&+3h_{ir}h_{rk}h_{jll}-2h_{ir}h_{jk}h_{rll}-h_{ijk}h_{rl}^2).
\end{eqnarray}
From (\ref{2.4.3}) and (\ref{2.4.5}), we have
\begin{eqnarray}\label{lem1.7}
h_{ijk}(h_{ijlk}-h_{ijkl})\langle X,e_l\rangle&=&h_{ijk}(h_{ir}R_{rjlk}+h_{rj}R_{rilk})\langle X,e_l\rangle\nonumber\\
&=&2h_{ijk}h_{ir}h_{jk}H_r-2h_{ijk}h_{ir}h_{rk}H_j.
\end{eqnarray}
Substituting (\ref{lem1.6}) and (\ref{lem1.7}) into (\ref{lem1.5}), we obtain
\begin{eqnarray}\label{lem1.8}
&&\frac{1}{2}(\Delta-\langle
X,\nabla\cdot\rangle)h_{ijk}^2=h_{ijk} (\Delta
h_{ijk}-h_{ijkl}\langle X,e_l\rangle)+h_{ijkl}^2\nonumber\\
&=&h_{ijk}(h_{ijlk}-h_{ijkl})\langle
X,e_l\rangle+2h_{ijk}^2+h_{ijkl}^2\nonumber\\
&&+(\lambda-H)h_{ijk}(h_{il}h_{jlk}+h_{jl}h_{ikl}+h_{kl}h_{jli})-h_{ijk}h_{il}h_{jl}H_k\nonumber\\
&&+h_{ijk}(h_{rkl}R_{rjil}+h_{rjk}R_{rlil}+2h_{rjl}R_{rikl}+2h_{ril}R_{rjkl}\nonumber\\
&&+h_{rij}R_{rlkl}
+h_{ir}(R_{rjkl})_l+h_{rj}(R_{rikl})_l+h_{lr}(R_{rjil})_k+h_{rj}(R_{rlil})_k)\nonumber\\
&=&2h_{ijk}h_{ir}h_{jk}H_r-2h_{ijk}h_{ir}h_{rk}H_j+2h_{ijk}^2+h_{ijkl}^2\nonumber\\
&&+3(\lambda-H)h_{ijk}h_{ijl}h_{kl}-h_{il}h_{jl}H_kh_{ijk}\nonumber\\
&&+h_{ijk}(6h_{rkl}h_{ri}h_{jl}-6h_{rkl}h_{rl}h_{ij}+3h_{rij}h_{rk}h_{ll}-3h_{rij}h_{rl}h_{kl}+3h_{ir}h_{rk}h_{jll}\nonumber\\
&&-2h_{ir}h_{jk}h_{rll}-h_{ijk}h_{rl}^2)\nonumber\\
&=&(2-|A|^2)h_{ijk}^2+h_{ijkl}^2+h_{ijk}(6h_{iu}h_{jv}h_{uvk}-3h_{iju}h_{uv}h_{kv})-\frac{3}{2}|\nabla|A|^2|^2\nonumber\\
& &+3\lambda h_{ijk}h_{ijl}h_{kl}.
\end{eqnarray}

(ii) It follows from the divergence theorem that
\begin{equation}\label{2.4.10-2}
\int_M \sum\limits_{i,\,j}(f_3)_{ij}h_{ij}\rho dM=-\int_M \sum\limits_{i,\,j}(f_3)_{i}(h_{ij}\rho)_{j} dM.
\end{equation}
By the condition $H=-X^N+\lambda$, we have
\begin{eqnarray}\label{2.4.10}
\sum\limits_{i,\,j}(h_{ij}\rho)_{j}&=&\nonumber\sum\limits_{i,\,j}h_{jji}\rho-\sum\limits_{i,\,j}h_{ij}\rho \langle e_i, X\rangle\\
&=&\nonumber-e_i(\langle \xi, X\rangle)\rho+\sum\limits_{i,\,j}\langle \nabla_{e_j}\xi, X\rangle\rho\\
&=&0.
\end{eqnarray}
This together with the divergence theorem implies
\begin{eqnarray}\label{2.4.10-3}
\int_M \sum\limits_{i,j,k}h_{ik}h_{kj}S_{ij}\rho dM&=&\nonumber-\int_M \sum\limits_{i,j,k}h_{ikj}h_{kj}S_{i}\rho dM\\
&=&-\frac{1}{2}\int_M |\nabla S|^2\rho dM.
\end{eqnarray}
Applying Ricci identity (\ref{2.4.3}), we get
\begin{equation}\label{2.4.9}
h_{ijij}-h_{jiji}=\mu_i\mu_j(\mu_i-\mu_j).
\end{equation}
Thus, we have
\begin{eqnarray}\label{2.4.10-1}
\frac{1}{3}\sum\limits_{i,j}h_{ij}(f_3)_{ij}&=&\nonumber\sum\limits_{i,k}h_{iikk}\mu_k\mu_{i}^2+2\sum\limits_{i,j,k}h_{ijk}^2\mu_i\mu_k\\
&=&\nonumber\sum\limits_{i,k}[h_{kkii}+(\mu_i-\mu_k)\mu_i\mu_k]\mu_k\mu_{i}^2+2B_2\\
&=&\nonumber\sum\limits_{i}(\frac{S_{ii}}{2}-\sum\limits_{j,k}h_{ijk}^2)\mu_{i}^2
+\sum\limits_{i,k}\mu_{i}^3\mu_{k}^2(\mu_i-\mu_k)+2B_2\\
&=&\sum\limits_{i,j,k}\frac{h_{ik}h_{kj}}{2}S_{ij}+Sf_4-f_{3}^2-(B_1-2B_2).
\end{eqnarray}
Substituting (\ref{2.4.10}), (\ref{2.4.10-3}) and (\ref{2.4.10-1}) into (\ref{2.4.10-2}), we obtain
\begin{equation}
\int_M (B_1-2B_2)\rho dM=\int_M[Sf_4-f_{3}^{2}-\frac{1}{4}|\nabla S|^2]\rho dM.
\end{equation}
\end{proof}

Combining (\ref{1.10}) and Lemma \ref{lem1}, we derive the following inequality.
\begin{eqnarray}
\int_M (B_1-2B_2) \rho d M & = &\nonumber \int_M \left( \frac{1}{2} G - \frac{1}{4} | \nabla S|^2 \right) \rho d M\\
& \leq &\nonumber \frac{2}{3} \int_M | \nabla^2 A |^2 \rho d M - \frac{1}{4}\int_M | \nabla S |^2 \rho d M\\
& = &\nonumber \frac{2}{3} \int_M (S-2) | \nabla A |^2 \rho d M + 2 \int_M (B_1-2B_2) \rho d M \\
& &+ \frac{3}{4} \int_M | \nabla S |^2 \rho d M - 2 \lambda \int_M C \rho dM.
\end{eqnarray}
This implies
\begin{eqnarray}\label{2.4.13}
\int_M (B_1-2B_2) \rho d M &\geq &\nonumber- \frac{2}{3} \int_M (S-2) | \nabla A |^2\rho d M\\
& &- \frac{3}{4} \int_M | \nabla S |^2 \rho d M + 2 \lambda \int_M C \rho d M.
\end{eqnarray}
For any $\sigma>0$, using Lemma 4.2 in \cite{DX} and Young's inequality, we have
\begin{equation}\label{2.4.14}
3(B_1-2B_2)\leq (S + C_1 G^{1 / 3})|\nabla A|^2\leq S|\nabla A|^2+\frac{1}{3}C_1\sigma^2 G+\frac{2}{3}C_1\sigma^{-1}|\nabla A|^3,
\end{equation}
where $C_1=\frac{2\sqrt{6}+3}{\sqrt[3]{21\sqrt{6}+103/2}}$.
Notice that
\begin{equation}\label{2.4.15}
-\int_M \nabla|\nabla A|\cdot\nabla S\rho dM=\int_M|\nabla A|\mathcal{L} S\rho dM.
\end{equation}
This together with (\ref{2.4.7}) implies
\begin{eqnarray}\label{2.4.16}
& &\nonumber\int_M|\nabla A|^3\rho dM\\
&=&\nonumber\int_M(F_\lambda+\frac{1}{2}\mathcal{L}|A|^2)|\nabla A|\rho dM\\
&=&\nonumber\int_MF_\lambda|\nabla A|\rho dM-\frac{1}{2}\int_M \nabla|\nabla A|\cdot\nabla S\rho dM\\
&\leq&\int_MF_\lambda|\nabla A|\rho dM+\epsilon\int_M|\nabla^2 A|^2\rho dM+\frac{1}{16\epsilon}\int_M |\nabla S|^2\rho dM,
\end{eqnarray}
for arbitrary $\epsilon>0$. We assume that $S$ satisfies the pinching condition $\beta_\lambda\leq S\leq\beta_\lambda+\delta$.
From (\ref{2.4.7-1}) and (\ref{2.4.7-2}), we have
\begin{eqnarray}\label{2.4.17}
\frac{1}{2}\int_M |\nabla S|^2\rho dM&=&\int_M(S-\beta_\lambda)F_\lambda\rho dM-\int_M(S-\beta_\lambda)|\nabla A|^2\rho dM \nonumber\\
&\leq&\int_M(-S+\beta_\lambda+\delta)|\nabla A|^2\rho dM.
\end{eqnarray}
For any $\kappa>0$, (\ref{2.4.7-4}) implies
\begin{eqnarray}\label{2.4.18}
\int_MF_\lambda|\nabla A|\rho dM&\leq&\nonumber2(\beta_\lambda+\delta)\kappa\int_MF_\lambda\rho dM\\
& &\nonumber+\frac{1}{8(\beta_\lambda+\delta)\kappa}\int_MF_\lambda|\nabla A|^2\rho dM\\
&\leq&\nonumber2(\beta_\lambda+\delta)\kappa\int_MF_\lambda\rho dM\\
& &\nonumber+\frac{1}{8(\beta_\lambda+\delta)\kappa}\int_M(|A|^2-\beta_\lambda+r_\lambda)|A|^2|\nabla A|^2\rho dM\\
&\leq&\nonumber2(\beta_\lambda+\delta)\kappa\int_M |\nabla A|^2\rho dM\\
& &+\frac{1}{8\kappa}\int_M(|A|^2-\beta_\lambda+r_\lambda)|\nabla A|^2\rho dM.
\end{eqnarray}
For $C$, we have the estimate
\begin{equation}\label{2.4.19}
|C|=|\sum\limits_{i,j,k}\mu_i h_{ijk}^2|\leq|A||\nabla A|^2.
\end{equation}
Combining (\ref{2.4.14}), (\ref{2.4.16}) and Lemma \ref{lem1}, we obtain
\begin{eqnarray}
&&3 \int_M (B_1-2B_2) \rho d M \nonumber\\
& \leq &\int_M \left( S | \nabla A |^2 + \frac{1}{3} C_1 \sigma^2 G + \frac{2}{3} C_1\sigma^{- 1} | \nabla A |^3 \right) \rho d M \nonumber\\
& \leq & \int_M S | \nabla A |^2 \rho d M +\frac{1}{3} C_1 \sigma^2 \int_M G \rho d M\nonumber\\
&  & + \frac{2}{3} C_1 \sigma^{- 1} \int_M F_{\lambda} | \nabla A | \rho dM + \frac{C_1}{24 \sigma \epsilon} \int_M | \nabla S |^2 \rho d M\nonumber\\
&  & + \frac{2}{3} C_1 \sigma^{- 1} \epsilon \int_M | \nabla^2 A |^2 \rho dM\nonumber\\
& = & \int_M S | \nabla A |^2 \rho d M + \frac{2}{3} C_1 \sigma^2 \int_M\left( B_1-2B_2 + \frac{1}{4} | \nabla S |^2 \right) \rho d M\nonumber\\
&  & + \frac{2}{3} C_1 \sigma^{- 1} \int_M F_{\lambda} | \nabla A | \rho dM + \frac{C_1}{24 \sigma \epsilon} \int_M | \nabla S |^2 \rho d M\nonumber\\
&  & + \frac{2}{3} C_1 \sigma^{- 1} \epsilon \int_M \left[ (S - 2) | \nabla A |^2 + 3(B_1-2B_2) + \frac{3}{2} | \nabla S |^2 - 3 \lambda C \right] \rho d M.
\end{eqnarray}
Hence
\begin{eqnarray}\label{2.4.19.a}
&&3 \theta \int_M (B_1-2B_2) \rho d M \nonumber\\
& \leq & \int_M \left[ S+\frac{2}{3} C_1 \sigma^{-1}\epsilon\left(S-2\right)\right] |\nabla A |^2 \rho dM\nonumber\\
&  & + \left( \frac{1}{6} C_1 \sigma^2 + \frac{C_1}{24 \sigma \epsilon} + C_1 \sigma^{- 1} \epsilon \right) \int_M | \nabla S |^2\rho d M\nonumber\\
&  & + \frac{2}{3} C_1 \sigma^{- 1} \int_M F_{\lambda} | \nabla A | \rho dM - 2 C_1 \sigma^{- 1} \epsilon \lambda \int_M C \rho d M,
\end{eqnarray}
where $\theta = 1 - \left( \frac{2}{9} C_1 \sigma^2 + \frac{2}{3} C_1 \sigma^{- 1}\epsilon \right)$.
When $\theta>0$, this together with (\ref{2.4.13}) implies
\begin{eqnarray}\label{2.4.20}
0 & \leq & \int_M [S+L_1(S-2)] | \nabla A |^2 \rho d M + L_2 \int_M |\nabla S |^2 \rho d M\nonumber\\
&  & + \frac{2}{3} C_1 \sigma^{- 1} \int_M F_{\lambda} | \nabla A | \rho dM - 2 \lambda (C_1 \sigma^{- 1} \epsilon + 3 \theta) \int_M C \rho d M,
\end{eqnarray}
where $L_1=\frac{2}{3}C_1\sigma^{-1}\epsilon+2\theta$, $L_2=\frac{1}{6}C_1\sigma^2+C_1\sigma^{-1}\epsilon+\frac{1}{24\epsilon}C_1\sigma^{-1}+\frac{9}{4}\theta$.
Substituting (\ref{2.4.17}), (\ref{2.4.18}) and (\ref{2.4.19}) into (\ref{2.4.20}), we obtain
\begin{eqnarray}\label{2.4.21}
0&\leq&\nonumber\int_M [S+L_1(S-2)]|\nabla A|^2\rho dM + 2L_2\int_M(-S+\beta_{\lambda}+\delta)|\nabla A|^2\rho dM\\
& &\nonumber+\frac{4}{3}C_1\sigma^{-1}(\beta_{\lambda}+\delta)\kappa\int_M|\nabla A|^2\rho dM
+\frac{C_1}{12\sigma\kappa}\int_M(S-\beta_{\lambda}+r_\lambda)|\nabla A|^2\rho dM\\
& &\nonumber+2|\lambda|(C_1\sigma^{-1}\epsilon+3\theta)\sqrt{\beta_{\lambda} + \delta}\int_M|\nabla A|^2\rho dM\\
&=&\nonumber \int_M \left[ \left(1+L_1-2L_2+\frac{C_1}{12\sigma\kappa}\right)(S-\beta_{\lambda})+(1+L_1)\beta_{\lambda}-2L_1+2L_2\delta \right.\\
& &\left.+ \frac{4}{3} C_1 \sigma^{-1} (\beta_{\lambda} + \delta) \kappa + \frac{C_1}{12\sigma\kappa} r_{\lambda}+2|\lambda|(C_1 \sigma^{- 1}\epsilon+3\theta)\sqrt{\beta_{\lambda}+\delta}\right]|\nabla A|^2\rho dM.
\end{eqnarray}
Denote by $\eta_\lambda=\frac{4}{3\sigma}C_1\tilde{r}_\lambda\kappa+(1+L_1)\tilde{r}_\lambda+\frac{C_1}{12\sigma\kappa} r_{\lambda}+2|\lambda|(C_1 \sigma^{- 1}\epsilon+3\theta)\sqrt{1+\tilde{r}_\lambda+\delta}$, $\tilde{r}_\lambda=\beta_{\lambda}-1$. Thus, (\ref{2.4.21}) is reduced to
\begin{eqnarray}\label{2.4.30}
0&\leq&\nonumber\left(1+L_1-2L_2+\frac{C_1}{12\sigma\kappa}\right)\int_M\left(S-\beta_{\lambda}\right)|\nabla A|^2\rho dM\\
& &+\left[1-L_1+\frac{4}{3\sigma}C_1\kappa+\left(\frac{4}{3\sigma}C_1\kappa+2L_2\right)\delta+\eta_\lambda\right]\int_M|\nabla A|^2\rho dM.
\end{eqnarray}
Let $\sigma=0.616, \epsilon=0.0577, \kappa=0.0434$. By a computation, we have
$$\theta > 0,\,\,\,1+L_1-2L_2+\frac{C_1}{12\sigma\kappa}<0,$$
$$1-L_1+\frac{4}{3\sigma}C_1\kappa<-0.452,\,\,\,2L_2+\frac{4}{3\sigma}C_1\kappa<8.03.$$
Take $\delta=1/18$. There exists an positive constant $\gamma$, such that $\eta_\lambda\leq0.005$ when $|\lambda|\leq\gamma$. Then the coefficients of the integral in (\ref{2.4.30}) are both negative. Therefore, $|\nabla A|\equiv0$. By a classification theorem due to Lawson \cite{Law}, $M$ must be $\mathbb{S}^k(r)\times\mathbb{R}^{n-k},\,\,1\leq k\leq n$. For $\lambda\neq0$,
the radius $r$ satisfies $\lambda=\frac{k}{r}-r$. Hence,
$$r=\frac{\sqrt{\lambda^2+4k}-\lambda}{2},$$
$$\mu_1=\ldots=\mu_k=\frac{1}{2k}(\sqrt{\lambda^2+4k}+\lambda),$$
where $\mu_k$is the $k$-th principal curvature of $M$.

We consider the following two cases:\\
(i) for $\lambda>0$, the squared norm of the second fundamental form $M$ satisfies
$$S_k=\sum\limits_{i=1}^{k}\mu_{i}^{2}=\frac{1}{2k}(\lambda^2+2k+|\lambda|\sqrt{\lambda^2+4k}).$$
Hence, $S_1=\beta_\lambda$. When $k\geq2$, $S_{k}<\beta_\lambda$.\\
(ii) for $\lambda<0$, by a computation, we have
$$S_{k}=\sum\limits_{i=1}^{k}\mu_{i}^{2}=\frac{1}{2k}(\lambda^2+2k-|\lambda|\sqrt{\lambda^2+4k}).$$
When $1\leq k\leq n$, $S_{k}<\beta_\lambda$.

Therefore, $\lambda>0$ and $M$ must be $\mathbb{S}\Big(\frac{\sqrt{\lambda^2+4}-|\lambda|}{2}\Big)\times\mathbb{R}^{n-1}$.
\end{proof}

\end{document}